\documentclass[11pt,a4paper]{amsart}

\usepackage{amssymb,amsfonts,amsmath,amstext,amsgen}
\usepackage[latin1]{inputenc}

\usepackage[norelsize,vlined,boxed]{algorithm2e}
\usepackage{tikz}

\usepackage{color}

\def\cqfd{
{\hfill
\kern 6pt\penalty 500
\raise -1pt\hbox{\vrule\vbox to 5pt{\hrule width 4pt
\vfill\hrule}\vrule}}
\break}

\newtheorem{theorem}{Theorem}
\newtheorem{proposition}[theorem]{Proposition}

\newtheorem{corollary}[theorem]{Corollary}

\newtheorem{remark}[theorem]{Remark}
\newtheorem{definition}[theorem]{Definition}

\makeatletter
\@namedef{subjclassname@2010}{%
\textup{2010} Mathematics Subject Classification}
\makeatother

\begin{document}


\title[]{Vectorial solutions to list multicoloring problems on graphs}
\author{Yves Aubry, Jean-Christophe Godin and Olivier Togni}
\address{Institut de Math\'ematiques de Toulon and Institut de Math\'ematiques de Luminy, Universit\'e du Sud Toulon-Var,  France,  and Laboratoire LE2I, Universit\'e de Bourgogne, France}

\email{yves.aubry@univ-tln.fr, godinjeanchri@yahoo.fr 
and olivier.togni@u-bourgogne.fr}



\date{\today}

\begin{abstract}
For a graph $G$ with a given list assignment $L$ on the vertices, we give an algebraical description of the set  of
all  weights $w$ such that $G$ is 
 $(L,w)$-colorable, called permissible weights.
Moreover, for a graph $G$ with a given list $L$ and a given permissible weight $w$, we describe the set  of all 
 $(L,w)$-colorings of $G$. By the way, we solve the {\sl channel assignment problem}. Furthermore, we describe the set of solutions to
 the {\sl on call problem}: when $w$ is not a permissible weight, we find all the nearest permissible weights $w'$. Finally, we give a solution to the non-recoloring problem keeping a given subcoloring.
 \end{abstract}

\maketitle


\section{Introduction}

It is convenient to model cellular data and communication networks as graphs with each node representing a base station in a cell in the network and edges representing geographical adjacency of cells. Moreover, we associate to each vertex in the graph a set of calls  in the cell served by the node corresponding to this vertex. 

The \textsl{channel assignment problem} (see  \cite{KatzalaNaghsineh1996}, \cite{JaumardMarcotteMeyer1999} and 
\cite{Hale1980}) is, at a given  time instant, to assign a number $w(v)$ of channels to each node $v$ in the network in
such a way that co-channel interference constraints are respected, and the total number of channels used over all nodes
in the network is minimized. 

The problem is related to the following graph multicoloring problem:
for a graph $G$ with a given list assignment $L$, find the weights $w$ such that $G$ is 
 $(L,w)$-colorable (see below for a precise definition of colorability). We will call such a weight $w$ a \textsl{permissible weight}.
 
 The purpose of this paper is to describe the set of all permissible weights $w$ and then to give a construction of  all 
 $(L,w)$-colorings of $G$. In particular we  solve the channel assignment problem which can be seen as: for a given graph $G$ and a given weight $w$ of $G$, find the 
weighted chromatic number $\chi(G,w)$ and furthermore find an $({\mathcal L}_{\chi(G,w)},w)$-coloring of $G$ (where 
${\mathcal L}_{\chi(G,w)}(v)=\{1,2,\ldots,\chi(G,w)\}$ for every vertex $v$ of $G$).

 Additionally, the description of the set of all permissible weights enable us to solve the {\sl on call problem}:  when
$w$ is not a permissible weight, we find all the nearest permissible weights $w'$. This is the situation we meet when
the network is exceptionally saturated as for the December 31st.

Finally,  we consider the {\sl non-recoloring problem} which arises when we want to extend a pre-coloring.

Note that although all our proofs are constructive (and hence algorithms can be derived from them), the
purpose of the present paper is not to compete with existing graph coloring algorithms such as those of Byskov
\cite{Byskov2003,Byskov2004} for the unweighted case or the one of Caramia and Dell'Olmo \cite{CaramiaOlmo2001} that
computes the weighted chromatic number. Our setting is more ambitious since we consider the list coloring problem on 
weighted graphs, for which, to our knowledge, no general algorithm exists.

 \bigskip
 
 The paper is organized as follows. We  develop a vectorial point of view in section \ref{A vectorial point of view}. In particular   for any color $x$  arising in $L$ we introduce the induced subgraph $G^x$ of $G$ whose vertices are those which have $x$ as a color in their list. After introducing a partial order in ${\mathbb N}^n$, we define hyperrectangles built on a finite set of vectors. Then, we define the set $\overrightarrow{W}_{max}(G,L)$ of a graph $G$ with a given list $L$: it is the set of sums of the  maximal independent vectors of all  the subgraphs $G^x$. This set will be a fundamental object in our results because it will be shown that this is the set of weight-vectors $\vec{w}$ which give a maximal $(L,w)$-coloring of $G$.
In Section \ref{Howtofindallthecolorings}   
  we consider a graph $G$ with a given list $L$ and a permissible weight $w$, and we describe the   set ${\mathcal C}(G,L,w)$ of all $(L,w)$-colorings of $G$. It gives an explicit answer to the channel assignment problem as shown in Section \ref{static}.
Then, Section \ref{Thesetofallpossibleweightvectors} is devoted to the second main result, namely, for a graph $G$ with
a given list $L$,  the characterization of the set $\overrightarrow{W}(G,L)$ of permissible weight-vectors $\vec w$
(i.e. such that $G$ is $(L,w)$-colorable). 
 We prove (Theorem \ref{W(G,L)}) that $\overrightarrow{W}(G,L)$ is the hyperrectangle of $\overrightarrow{W}_{max}(G,L)$.
Section \ref{The set of all minimal rejected weights} is concerned with the {\sl on call problem}. We describe the set
of weights which give an answer to this problem.
Section \ref{nonrecoloring} focuses on the non-recoloring problem and Section \ref{Algorithmic considerations} deals
with some algorithmic considerations.

\bigskip
\bigskip

In the paper, all the graphs are simple, undirected and with a finite number of vertices.

If $G$ is a graph, we denote by $V(G)$  the set of its vertices and by $E(G)$ the set of its edges. 

A list assignment (called simply a list) of $G$ is a map $L : V(G) \rightarrow \mathcal{P}({\mathbb N})$: to each vertex $v$ of $G$, we associate a finite set of integers which can be viewed as possible colors that can be chosen on $v$. 

If $a$ is an integer $\geq 1$, we define  the $a$-uniform list ${\mathcal L}_a$ of $G$ by: for every vertex $v$ of $G$,
$${\mathcal L}_a(v)=\{1,2,\ldots,a\}.$$

A weight of $G$ is a map $w : V(G) \rightarrow {\mathbb N}$: to each vertex $v$ of $G$, we associate an integer which can be viewed as the number of wanted colors on $v$. 

The cardinal of a finite set $A$ will be denoted by $\vert A\vert$.

We recall in the next definition what we mean exactly by an $(L,w)$-coloring of a graph, the central notion of this paper.

\begin{definition}
Let $G$ be a graph with a given list $L$ and a given weight $w$.
An $(L,w)$-coloring $C$ of a graph $G$ is a map $C$ that associate to each vertex $v$ exactly $w(v)$ colors from $L(v)$ and such that adjacent vertices receive disjoints color sets, i.e. for all $v \in V(G)$:
$$ C(v) \subset L(v) \ , \ \vert C(v) \vert = w(v)$$
and for all $vv' \in E(G) :$
$$ C(v) \cap C(v') = \emptyset \ .$$

We say that  $G$ is $(L,w)$-colorable  if there exists an $(L,w)$-coloring  of $G$. 
\end{definition}


\section{A vectorial point of view}
When dealing with graphs with $n$ vertices, we will work with vectors with integer coordinates in the vector space ${\mathbb R}^n$.
\label{A vectorial point of view}

\subsection{The vectorial decomposition}

For any $n \in {\mathbb N}$, let us set 
$${\mathbb N}^n:={\mathbb N}\vec{e_1}+\cdots{\mathbb N}\vec{e_n}$$
 where $(\vec{e_i})_{1\leq i\leq n}$ is a basis of the ${\mathbb R}$-vector space ${\mathbb R}^n$.
 For any vector $\vec{x}=\sum_{i=1}^nx_i\vec{e}_i$ we consider the norm $\| \vec{x} \| = \sum_{i=1}^n x_i$.

Let $G$ be a graph with $n$ vertices $v_1,\ldots,v_n$, let $L$ be a list of $G$ and let $w$ be a weight of $G$.
Let us introduce some notation.

For any subset $N$ of the set $V(G)$ of vertices of $G$, we associate the vector 
$$\vec{N}=\sum_{i=1}^n\lambda_i\vec{e_i}\in{\mathbb N}^n$$
 defined by: $\lambda_i=1$ if $v_i\in N$ and 0 if $v_i\not\in N$.

We define the set of all colors of $L$, by:
$${\overline L}:=\bigcup_{v\in V(G)}L(v) \ ,$$
and we define  $\tilde L\in (\mathcal{P}({\mathbb N}))^n$ the $n$-tuple  of sets:
$$\tilde{L}:=(L(v_1),\ldots, L(v_n)).$$

\noindent
For any list $L'$ of $G$, we define the union-list  $\tilde{L}\ \tilde{\cup}\ \tilde{L'}$, by:
$$\tilde{L}\ \tilde{\cup}\ \tilde{L'}:=(L(v_1)\cup L'(v_1),\ldots, L(v_n)\cup L'(v_n)).$$

\noindent
For any   $(L,w)$-coloring $C$ of $G$, we define its weight-vector  $\vec{w}(C)\in{\mathbb N}^n$ by:
$$\vec{w}(C) := \sum_{i=1}^{n} \vert C(v_i) \vert \ \vec{e}_i=\sum_{i=1}^nw(v_i)\vec{e_i}.$$

\begin{definition}
For any color $x\in\overline{L}$ and any   $(L,w)$-coloring $C$ of $G$, we define the $x$-color sublist $C^x$ as the list of the graph $G$ defined by: for any $v_i\in V(G)$, $C^x(v_i)=\{x\}$ if $x\in C(v_i)$, and $C^x(v_i)=\emptyset$ otherwise.\\
\end{definition}

The following proposition gives the decomposition of any coloring in terms of its $x$-color sublists.

\begin{proposition}[]
\label{decomposition}
For any graph $G$, any list $L$ of $G$ and any weight $w$ of $G$, if $C$ is  an $(L,w)$-coloring of $G$, then:
$$\tilde{C}=\tilde{\bigcup}_{x\in\overline{L}}\tilde{C^x} \ , \ and \ \vec{w}(C)=\sum_{x\in\overline{L}}\vec{w}(C^x).$$
\end{proposition}

\begin{proof}
For any color $x \in \overline{L}$, consider the $x$-color sublist $C^x$ of the graph $G$ defined above.
 By construction we have, for any vertex $v \in V(G)$, $C(v)=\bigcup_{x \in \overline{L} } C^x(v)$, therefore $\tilde{C}=\tilde{\bigcup}_{x\in\overline{L}}\tilde{C^x}$. For any $x,y \in \overline{L}$ such that $x \not= y$, we have for any vertex $v \in V(G)$, $C^x(v) \cap C^y(v) = \emptyset$, therefore $\vec{w}(C)=\sum_{x\in\overline{L}}\vec{w}(C^x)$.
\end{proof}

\begin{definition}
For any graph $G$ and any color $x\in\overline{L}$, we define the $x$-color subgraph $G^x$ to be the induced subgraph of $G$ defined by: 
$v\in V(G^x)$ if  and only if  $x\in L(v).$
\end{definition}

Remark that if $L^x$ denotes the list of the graph $G^x$ defined by: $L^x(v)=\{x\}$ for any $v\in V(G^x)$, then $C^x$ is an $(L^x,w(C^x))$-coloring of $G^x$.


\subsection{Hyperrectangles}

In order to define the hyperrectangles, let us introduce a (partial) order in ${\mathbb N}^n$.

For any vectors $\vec{x} \ , \vec{y} \  \in{\mathbb N}^n$, we say that 
$$\vec{y}=\sum_{i=1}^ny_i\vec{e_i}\leq \vec{x}=\sum_{i=1}^nx_i\vec{e_i}$$
 if and only if  $y_i\leq x_i$ for all $1\leq i\leq n$. 
 
 For any vector $\vec{x}\in{\mathbb N}^n$, we define the hyperrectangle of $\vec{x}$ (see Fig.~\ref{fig1}) by:
$$R(\vec{x}):=\{  \vec{y}\in{\mathbb N}^n \ \mid \  \vec{y}\leq \vec{x}\} \ .$$

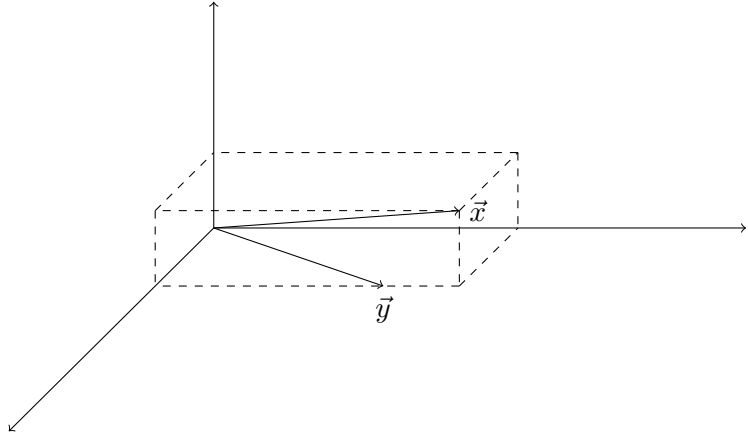
\begin{figure}[Ht]
\begin{tikzpicture} 

\draw[->] (0,0,0) -- (7,0,0); 
\draw (7,0,0) node[right] {$$}; 

\draw [->] (0,0,0) -- (0,3,0); 
\draw (0,3,0) node[above] {$$}; 

\draw [->] (0,0,0) -- (0,0,7); 
\draw (0,0,7) node[left] {$$}; 

\draw (4,1,2) node[right]{$\vec{x}$} ; 

\draw[->] (0,0,0) -- (4,1,2); 

\draw[->] (0,0,0) -- (3,0,2); 

\draw (3,0,2) node[below]{$\vec{y}$} ; 

\draw [dashed] (4,1,2) -- (4,1,0) node[below] {$$}; 
\draw [dashed] (4,1,2) -- (4,0,2) node[left] {$$}; 
\draw [dashed] (4,1,2) -- (0,1,2) node[left] {$$}; 

\draw [dashed] (4,0,2) -- (4,0,0) node[below] {$$}; 
\draw [dashed] (4,0,2) -- (0,0,2) node[left] {$$}; 

\draw [dashed] (0,1,2) -- (0,0,2) node[below] {$$}; 
\draw [dashed] (0,1,2) -- (0,1,0) node[left] {$$}; 

\draw [dashed] (4,1,0) -- (4,0,0) node[below] {$$}; 
\draw [dashed] (4,1,0) -- (0,1,0) node[left] {$$};

\end{tikzpicture} 
\caption{\label{fig1} A vector $\vec{y}$ in the hyperrectangle of $\vec{x}$.}
\end{figure}

For a finite set of vectors $X=\{\vec{x_1},\ldots,\vec{x_k}\}$ with $\vec{x_i}\in{\mathbb N}^n$, $1\leq i\leq k$, we define the hyperrectangle of $X$, by:
$$R(X)=R( \{ \vec{x_1},\ldots,\vec{x_k} \} ):=\bigcup_{i=1}^kR(\vec{x_i}) \ .$$



\subsection{The set of maximal independent sets of a graph}

Recall that an independent (or stable) set $S$ of a graph $G$ is a subset of $V(G)$ such that $ vv' \notin E(G)$ for any $v,v' \in S$. A maximal independent set of a graph $G$ is an independent set $S$ of $G$ of maximal cardinality, i.e. $S$ is not the proper subset of another independent set.

Now, let $G$ be a graph and $L$ be a list of $G$.
For any color $x$ arising in $L$, we consider the maximal independent sets of the subgraphs $G^x$. They are subsets of $V(G)$ and we consider their associated vectors. Summing all these vectors, we obtain the following definition of the  set  $\overrightarrow{W}_{max}(G,L) $ which will be central in this paper.

For any sets $X_1,\ldots, X_k$  of vectors of ${\mathbb R}^n$, we recall that the vectorial sum is defined by:
$$\sum_{i=1}^kX_i := \{\vec{x}_1+\cdots +\vec{x}_k\ \vert \  \vec{x}_1\in X_1,\ldots, \vec{x}_k \in X_k\} \ .$$

Recall also that if $N$ is a subset of $V(G)=\{v_i,\ldots,v_n\}$, then
the vector $\vec N$ is defined by 
$\vec{N}=\sum_{i=1}^n\lambda_i\vec{e_i}$ with
 $\lambda_i=1$ if $v_i\in N$ and 0 otherwise.
 
 If $x \in \overline{L}$ and $H$ is a subgraph of $G$ we define:
 $$\overrightarrow{MIS}(H) := \{\vec{N} \in{\mathbb N}^{\vert V(G) \vert} \mid N \ {\rm is \ a \ maximal \ independent \ set \ of}\  H \}.$$

\begin{definition}
For any graph $G$ and any list $L$ of $G$, we define

$$\overrightarrow{W}_{max}(G,L) := \sum_{x \in \overline{L} } \overrightarrow{MIS}(G^x).$$

\end{definition}


\section{The set of all colorings} 
\label{Howtofindallthecolorings}

For any graph $G$,  any list $L$ of $G$ and any weight $w$ of $G$, we consider the   set ${\mathcal C}(G,L,w) $ of all $(L,w)$-colorings of $G$.

The purpose of this section is to give a description of this set. In order to do it, we define the maximal $(L,w)$-colorings set:

$${\mathcal C}(G,L,w)_{max} := \{ C \in  {\mathcal C}(G,L,w)   \mid \forall x \in \overline{L}, \ \vec{w}(C^x) \in  \overrightarrow{MIS}(G^x)\}.$$
>From Proposition \ref{decomposition}, we have the following property:

\begin{equation*}
{\mathcal C}(G,L,w)_{max} \not= \emptyset \iff \vec{w} \in \overrightarrow{W}_{max}(G,L).\hskip2cm (\ast)
\end{equation*}

For any $C \in {\mathcal C}(G,L,w)$ and any $\vec{d} \in \mathbb{N}^n$, we define the  subcoloring set:
$$ {\mathcal C}^-(C,\vec{d}) := \{ C' \mid  \forall v \in V(G):C'(v) \subset C(v), \ and \ \vec{w}(C')=\vec{w}(C)-\vec{d}  \}.$$

\begin{theorem}
\label{propdecompositionC}
For any graph $G$, any list $L$ of G and any permissible weight $w$  of $G$ i.e. such that $G$ is $(L,w)$-colorable, we have: 
$$ {\mathcal C}(G,L,w)= \bigcup \limits_{\underset{\vec{w}' \geq\vec{w}}{\vec{w}' \in \overrightarrow{W}_{max}(G,L)}} \  \bigcup_{C' \in {\mathcal C} (G,L,w')_{max}  } {\mathcal C}^{-}(C',\vec{w}' - \vec{w} ) .$$
\end{theorem}

\bigskip

\begin{proof}
Let $G$ be a graph, $L$ a list of $G$ and $w$ a permissible weight.  If $C \in {\mathcal C}(G,L,w)$, then for any $x \in \overline{L}$, we define the subset $N^x$ of $V(G)$ such that $\vec{N}^x=\vec{w}(C^x)$. Therefore $N^x$ is an independent set of $G^x$, and there exists  a maximal independent set $S^x$ of $G^x$ such that $N^x \subset S^x$. We construct a list $C'$ of $G$ such that for any $x\in \overline{L}$, $\vec{w}(C'^x)=\vec{S}^x$. By construction $\vec{w}' = \vec{w}(C')=\sum_{x \in \overline{L}} \vec{w}(C'^x) \in \overrightarrow{W}_{max}(G,L)$, then $C' \in {\mathcal C}(G,L,w')_{max} $. Since $N^x \subset S^x$, then $\vec{w}=\sum_{x \in \overline{L} } \vec{w}(C^x) \leq \sum_{x \in \overline{L} } \vec{w}(C'^x)=\vec{w}'$, and by construction $C \in {\mathcal C}^{-}(C',\vec{w}' - \vec{w} )$. Hence we have the first inclusion.\\
If $\vec{w}' \in \overrightarrow{W}(G,L)$ such that $\vec{w}' \geq \vec{w}$, and $C' \in {\mathcal C} (G,L,w')_{max} $. For any $C'' \in {\mathcal C}^{-}(C',\vec{w}' - \vec{w} )$, we have $\vec{w}(C'')=\vec{w}'-( \vec{w}' - \vec{w})=\vec{w} $, therefore $C'' \in {\mathcal C}(G,L,w)$, and we have the reverse inclusion.
\end{proof}


\section{The set of all weights} 
\label{Thesetofallpossibleweightvectors}

Now, we can state the main result of the paper: the set $\overrightarrow{W}(G,L)$ of all possible weights vectors $\vec{w}$ such
that $G$ is $(L,w)$-colorable is given by the hyperrectangle of the set $\overrightarrow{W}_{max}(G,L)$:

\begin{theorem}[]
\label{W(G,L)}
For any graph $G$ and any list $L$ of $G$:
$$\overrightarrow{W}(G,L)=R(\overrightarrow{W}_{max}(G,L)).$$
In other words, the graph $G$ is $(L,w)$-colorable if and only if the vector $\vec{w}$ belongs to the hyperrectangle constructed on the maximal independent sets of the subgraphs $G^x$'s.
\end{theorem}

\begin{proof}
Let $G$ be a  graph and $L$ be a  list of $G$. If $\vec{w}\in \overrightarrow{W}(G,L)$ then ${\mathcal C}(G,L,w) \not= \emptyset$. By  Theorem \ref{propdecompositionC}, there exists $\vec{w}' \in \overrightarrow{W}_{max}(G,L)$ such that $\vec{w}' \geq\vec{w}$. Then $\overrightarrow{W}(G,L) \subset R(\overrightarrow{W}_{max}(G,L))$.\\
If $\vec{0} \not= \vec{w} \in R(\overrightarrow{W}_{max}(G,L))$, by construction there exists $\vec{w}' \in \overrightarrow{W}_{max}(G,L)$ such that $\vec{w}' \geq\vec{w}$. By Theorem \ref{propdecompositionC} and Property $(\ast)$
we have ${\mathcal C}(G,L,w) \not= \emptyset$, therefore $\vec{w}\in \overrightarrow{W}(G,L)$, and since $\vec{0} \in \overrightarrow{W}(G,L)$ we have the reverse inclusion.
\end{proof}



Theorem \ref{W(G,L)} can be written in a nice way in the particular case of a $a$-uniform list ${\mathcal L}_a$ (i.e.
${\mathcal L}_a(v)=\{1,2,\ldots,a\}$ for any $v$).

\begin{corollary}
\label{listeidentique}
Let $G$ be a graph with $m$ maximal independent sets $S_1,\ldots,S_m$. Then, we have:
$$\overrightarrow{W}(G,{\mathcal L}_a)=\{\sum_{i=1}^mx_i\vec{S}_i,\ \ {\rm with}\ x_i\in{\mathbb N}\ {\rm and} \ \sum_{i=1}^mx_i\leq a\}.$$

\end{corollary}

\begin{proof}
By definition we have $\overrightarrow{W}_{max}(G,{\mathcal L}_a)=\sum_{x \in \overline{L} } \overrightarrow{MIS}(G^x)$.
In the case of the $a$-uniform list ${\mathcal L}_a$, the subgraph $G^x$ is equal to $G$ for any color $x\in
\overline{L}$ thus  $\overrightarrow{W}_{max}(G,{\mathcal L}_a)=\sum_{i=1}^a \overrightarrow{MIS}(G)$. Since
$\overrightarrow{MIS}(G)=\{\vec{S}_1,\ldots,\vec{S}_m\}$, we obtain
$$\overrightarrow{W}_{max}(G,{\mathcal L}_a)=\{\sum_{i=1}^mx_i\vec{S}_i,\ \ {\rm with}\ x_i\in{\mathbb N}\ {\rm and} \ \sum_{i=1}^mx_i=a\}.$$
Then Theorem \ref{W(G,L)} concludes the proof taking the hyperrectangle of $\overrightarrow{W}_{max}(G,{\mathcal L}_a)$.
\end{proof}


\section{The static channel assignment problem} 
\label{static}

Let $G$ be a graph and $w$ a weight of $G$. 
Then, the weighted chromatic number $\chi(G,w)$ of $G$ associated to the weight $w$ is defined to be the smallest integer $a$ such that $G$ is $({\mathcal L}_a,w)$-colorable: it is the minimum number 
of colors for which there exists a proper weighted $a$-coloring (see \cite{CaramiaFiala2004} for a recent study of this number).
In the particular case where the weight $w$ is defined by $w(v)=1$ for all vertex $v$ of $G$ (i.e. if the weight-vector $\vec{w}=(1,\ldots,1)$) then
$\chi(G,w)$ is nothing but the chromatic number $\chi(G)$ of $G$.

The static channel assignment problem can then be viewed as follows: for a given graph $G$ and a given weight $w$ of $G$, find the 
weighted chromatic number $\chi(G,w)$ and furthermore find an $({\mathcal L}_{\chi(G,w)},w)$-coloring of $G$.

Our strategy to solve the static assignment problem is the following one.

Firstly, we want to find the weighted chromatic number $\chi(G,w)$. We consider the list ${\mathcal L}_1$ of $G$ and we compute 
$$\overrightarrow{W}_{max}(G,{\mathcal L}_1) := \sum_{x \in \overline{{\mathcal L}_1} } \overrightarrow{MIS}(G^x)= \overrightarrow{MIS}(G).$$

We check whether $\vec{w}\in R(\overrightarrow{W}_{max}(G,{\mathcal L}_1))$, i.e. if there exists $\vec{w}'\in
\overrightarrow{W}_{max}(G,{\mathcal L}_1)$ such that $\vec{w}\leq \vec{w}'$.

If the answer is positive then we have $\chi(G,w)=1$ by Theorem \ref{W(G,L)}. Otherwise, we compute $\overrightarrow{W}_{max}(G,{\mathcal L}_2)$.                        

\begin{remark}
Note that, since $\chi(G,w)\geq \frac{\| \vec{w}\| }{\alpha(G)}$ where $\alpha(G)$ is the independence number of $G$ i.e. the size of the largest independent set of $G$, we can begin with the computation of $\overrightarrow{W}_{max}(G,{\mathcal L}_{\frac{\| \vec{w}\| }{\alpha(G)}})$.

Note also that since our lists ${\mathcal L}_a$ are $a$-uniform lists, we have by Corollary \ref{listeidentique} an easy
description of the set $\overrightarrow{W}_{max}(G,{\mathcal L}_a)$.

\end{remark}

Hence we find $a= \chi(G,w)$ as soon as we find $\vec{w}'\in \overrightarrow{W}_{max}(G,{\mathcal L}_a)$ such that $\vec{w}'\geq \vec{w}$ and such that

$$\vec{w}'=\sum_{i=1}^a \vec{z}_i \in \sum_{i=1}^a \overrightarrow{MIS}(G).$$

By Theorem \ref{propdecompositionC}, we can construct an $({\mathcal L}_a,w')$-coloring $C'$ of $G$.

But we have by Proposition \ref{decomposition}:

$$\tilde{C'}=\tilde{\bigcup}_{x=1}^a\tilde{C'^x} \ \ {\rm and}\  \ \vec{w}(C'^x)=\vec{z}_x.$$

By Theorem  \ref{propdecompositionC}, it is sufficient to take one $C\in  {\mathcal C}^-(C',\vec{w}'-\vec{w})$ to get a  $({\mathcal L}_a,w)$-coloring  of $G$.



\section{The on call problem} 
\label{The set of all minimal rejected weights}

The {\sl on call problem} can be modelized as follows : for $\vec{w} \notin \overrightarrow{W}(G,L)$, find $\vec{w}^{*} \in \overrightarrow{W}(G,L)$ such that $\vec{w}^{*}\leq  \vec{w}$ and $\| \vec{w}-\vec{w}^{*} \|$ is minimal.
We define the vector $\min(\vec{x},\vec{y})$ as the vector  $\vec{z}$ such that $z_i=\min(x_i,y_i)$ for all $i \in \{1,\dots,n\}$. Now for a fixed vector $\vec w$ we define the set
$$\min(\vec{w}, \overrightarrow{W}_{max}(G,L)):=\{ \min(\vec{w},\vec{w}')  \mid  \vec{w}' \in  \overrightarrow{W}_{max}(G,L)\}.$$

\begin{theorem}[]
\label{rejected}
Let $G$ be a graph,  $L$ a list of $G$ and $w$ a weight of $G$. Then  
$\vec{w}^{*}$ is a solution to the {\sl on call problem} if and only if
$\vec{w}^{*}$ is a vector of $\min(\vec{w}, \overrightarrow{W}_{max}(G,L))$ of maximal norm i.e. 
$\vec{w}^{*}\in \min(\vec{w}, \overrightarrow{W}_{max}(G,L))$ such that $ \forall\vec{w}'\in \min(\vec{w}, \overrightarrow{W}_{max}(G,L))$, we have 
$\| \vec{w}^{*} \| \geq \| \vec{w}' \|.$
\end{theorem}

\begin{proof}
For any graph $G$, any list $L$ of $G$, and any weight $w$ of $G$, if $\vec{w}^{*} $ is a solution to the {\sl on call problem}, then $\vec{w}^{*} \in \overrightarrow{W}(G,L)$. By Theorem \ref{W(G,L)}, there exists $\vec{w}' \in \overrightarrow{W}_{max}(G,L)$ such that $\vec{w}^{*} \leq \vec{w}'$. Since by definition, $\vec{w}^{*} \leq \vec{w}$,  then   $\vec{w}^{*} \leq  \min(\vec{w},\vec{w}') \leq \vec{w}$, and therefore $ \| \vec{w}- \min(\vec{w},\vec{w}') \| \leq \| \vec{w}- \vec{w}^{*} \|$. 

By the minimality property, we have $\vec{w}^{*}= \min(\vec{w},\vec{w}') \in \min(\vec{w},\overrightarrow{W}_{max}(G,L)$. Moreover, for any $\vec{w}'' \in \min(\vec{w},\overrightarrow{W}_{max}(G,L) $, since $\vec{w}\geq \vec{w}^{*}$ and $\vec{w}>\vec{w}''$,
 we have 
$$ \| \vec{w}- \vec{w}^{*} \|  =
\| \vec{w}\| - \|\vec{w}^{*} \|
\leq 
\| \vec{w}\| - \|\vec{w}'' \|
=
 \| \vec{w}- \vec{w}'' \|,$$
thus $\| \vec{w}^{*} \| \geq \| \vec{w}'' \| $.\\

Let $\vec{w}^{*}$ be a solution to the {\sl on call problem} and let
 $\vec{w}^{**} \in \min(\vec{w}, \overrightarrow{W}_{max}(G,L))$ such that $ \forall\vec{w}'\in \min(\vec{w}, \overrightarrow{W}_{max}(G,L))$, we have 
$\| \vec{w}^{**} \| \geq \| \vec{w}'\|$. We have $\vec{w}^{**}=\min(\vec{w},\vec{u})$ with $\vec{u}\in\overrightarrow{W}_{max}(G,L)$, hence $\vec{w}^{**}\in R(\overrightarrow{W}_{max}(G,L))$ thus $\vec{w}^{**}\in \overrightarrow{W}(G,L)$ by Theorem \ref{W(G,L)}.

 By construction $\vec{w} \geq \vec{w}^{**} \in \overrightarrow{W}(G,L)$ such that $\| \vec{w}-\vec{w}^{**}\|=\| \vec{w}-\vec{w}^{*} \|$, therefore $\vec{w}^{**}$ is also a solution to the {\sl on call problem}.
\end{proof}



\section{Non-recoloring problem}
\label{nonrecoloring}

Let $G$ be a graph, $w_0$ a weight of $G$ and $a_0=\chi(G,w_0)$ the weighted chromatic number of $G$ associated to $w_0$. Now, consider the list  ${\mathcal L}_{a_0}$ of $G$ and $C_0$ an $({\mathcal L}_{a_0},w_0)$-coloring of $G$. Finally let $w$ be a weight  of $G$ such that $\vec{w} \geq \vec{w}_0$.

The non-recoloring problem is to find the smallest $a$ (denoted by $\chi(G,w,C_0)$) such that there exists an $({\mathcal L}_a,w)$-coloring $C$ of $G$ such that $C_0$ is a subcoloring of $C$. In other words 
$$C_0 \in \mathcal{C}^-(C,\vec{w}-\vec{w}_0).$$

We define the set
 $\overrightarrow{W}_{max}(G,{\mathcal L}_{a_0},C_0)$ to be the set of vectors $\vec{w}_1 \in \overrightarrow{W}_{max}(G,L_{a_0})$ such that there exists an $({\mathcal L}_{a_0},w_1)$-coloring $C_1$ of $G$ such that 
$C_0 \in \mathcal{C}^-(C_1,\vec{w_1}-\vec{w}_0)$.

By definition we have
$\overrightarrow{W}_{max}(G,{\mathcal L}_{a_0},C_0) \subset \overrightarrow{W}_{max}(G,{\mathcal L}_{a_0})$

Then we set
$$\vec{\delta}(G,{\mathcal L}_{a_0},C_0,w) := \{ \vec{w} - \min(\vec{w},\vec{w}_1) \mid \vec{w}_1 \in \overrightarrow{W}_{max}(G,{\mathcal L}_{a_0},C_0) \}.$$

\begin{theorem}
\label{thnonrec}
If $G$ is a graph with the above notation then 
$$ \chi(G,w,C_0) \le a_0 + \min_{\vec{w}' \in \vec{\delta}(G,{\mathcal L}_{a_0},C_0,w) } \chi(G,w').$$
\end{theorem}

\begin{proof}
Let $a= a_0 + \min_{\vec{w}' \in \vec{\delta}(G,{\mathcal L}_{a_0},C_0,w) } \chi(G,w')$. Let  $\vec{w}^* \in \vec{\delta}(G,{\mathcal L}_{a_0},C_0,w) $ such that
$$ \chi(G,w^*)=\min_{\vec{w}' \in \vec{\delta}(G,{\mathcal L}_{a_0},C_0,w) } \chi(G,w').$$
Then there exists  $\vec{w}^{**} \in \overrightarrow{W}_{max}(G,{\mathcal L}_{a_0},C_0) $ such that
$$ \vec{w}^* = \vec{w} - \min(\vec{w},\vec{w}^{**}).$$

Since $\vec{w} \geq \vec{w_0} $ and by definition of $\vec{w}^{**}$, the vector $\vec{w}_{3} := \min(\vec{w},\vec{w}^{**}) \geq \vec{w}_0$, we can thus construct a   $({\mathcal L}_{a_0},w_{3})$-coloring $C_3$ of $G$ such that   $C_0 \in \mathcal{C}^-(C_3,\vec{w}_{3}-\vec{w}_0)$. By our construction, choosing the colors in the set $[a_0+1,a]$, we can construct an $({\mathcal L}_{a-a_0},w^*)$-coloring $C_4$ of $G$. Finally we construct $C$ such that 
 $\tilde{C}:=\tilde{C}_3 \tilde{\cup} \tilde{C}_4$, thus  $C$ is an $({\mathcal L}_a,w)$-coloring of $G$ and thus
$\chi(G,w,C_0) \leq a$.
\end{proof}

We believe that the inequality in Theorem~\ref{thnonrec} can be replaced by an equality but we have not been able to
prove it.

\section{Algorithmic considerations}
\label{Algorithmic considerations}

Proofs of Theorems ~\ref{propdecompositionC}, \ref{W(G,L)}, \ref{rejected} and \ref{thnonrec} are all constructive and
thus algorithms
can be derived from them. But, of course, since the problems considered are all NP-complete, there is little hope for
polynomial complexity.

By Theorem ~\ref{W(G,L)}, the problem of knowing, for a given graph $G$ with a list assignment $L$, if a given weight
$\vec{w}$ is permissible (i.e. if $G$ is $(L,w)$-colorable) reduces to that of constructing the set
$\overrightarrow{W}_{max}(G,L)$ (if the set is generated vector by vector, then we can stop as soon as a vector
$\vec{w'}$ with $w'\ge w$ is output).
\smallskip

\begin{algorithm}[ht]
\label{algo1}
\SetKwFunction{FW}{Function $\overrightarrow{W}_{max}(G, L)$}
\SetKwFunction{FS}{Function VecSum($S,S',c$)}
\SetKwFunction{S}{VecSum($W, M, i$)}

\KwData{graph $G$; list $L$;}
\KwResult{the vector set $\overrightarrow{W}_{max}(G,L)$}
\Begin{
M $\leftarrow$ $MIS(G)$\;
W $\leftarrow$ $\{(0,0,\ldots, 0)\}$\;
 
 \For{$i \in{\bar L}$}{
   W $\leftarrow$ \S \;
     }
 return W\;
}

\BlankLine
\BlankLine
\FS
\tcp*[l]{computes the vectorial sum of $S$ and the restriction of $S'$ to $G^c$}
\Begin{
T $\leftarrow$ $\emptyset$\;
 \For{$s\in S$}{
\For{$s'\in S'$}{
$s'^{c}$ $\leftarrow$ the restriction of $s'$ to $G^c$\;
   \If{$s'^{c}$ is maximal in $G^c$}{T $\leftarrow T\cup \{s+s'^{c}\}$;}}}
 return T\;
}
\caption{computing ${W}_{max}(G,L)$}
\end{algorithm}
\smallskip

For the problem of given a graph $G$ of order $n$ with a list $L$, listing the vector set
$\overrightarrow{W}_{max}(G,L)$, a
good measure of performance is the time required compared with the size $m$ of $\overrightarrow{W}_{max}(G,L)$. The set
$\overrightarrow{MIS}(G)$ of all maximal independent vectors of $G$ can be constructed in time within a polynomial
factor of its
size~\cite{Tsu77} (that can be as large as $n^\frac{n}{3}=(1.44225)^n$). Then, for any color $x$, the set
$\overrightarrow{MIS}(G^x)$ can be computed 'on the fly' by checking if the restriction of each vector of
$\overrightarrow{MIS}(G)$ to $G^x$ is maximal (clearly, for any independent set of $G$, its restriction to $G^x$ is
also an independent set). Checking the maximality can be done in $O(n²)$ operations. Algorithm~\ref{algo1} describes the
steps to compute the set $\overrightarrow{W}_{max}(G,L)$. Its time complexity in the worst case is in $O(m^\ell)$, were
$\ell=|{\bar L}|$.







\end{document}